\documentclass[12pt]{amsart}
\usepackage{amsmath,amssymb,amsfonts,amsthm,amsopn}
\usepackage{graphicx,tikz}
\usepackage[all]{xy}
\usepackage{amsmath}
\usepackage{multirow, longtable, makecell, caption, array,enumitem}
\usepackage{amssymb}
\usepackage{mathtools}
\usepackage{bookmark}
\usepackage{hyperref}

\calclayout

\newtheorem{theorem}{Theorem}[section]
\newtheorem{lemma}[theorem]{Lemma}
\newtheorem{corollary}[theorem]{Corollary}
\newtheorem{proposition}[theorem]{Proposition}
\newtheorem{definition}[theorem]{Definition}

\newtheorem{remark}[theorem]{Remark}

\theoremstyle{definition}

\usepackage{cellspace} %
\def\bR{\mathbb{R}}
\def\bC{\mathbb{C}}
\def\bN{\mathbb{N}}
\def\bZ{\mathbb{Z}}

\def\cF{\mathcal{F}}
\def\cS{\mathcal{S}}

\def\cL{\mathcal{L}}

\def\cM{\mathcal{M}}
\def\cW{\mathcal{W}}
\def\cH{\mathcal{H}}
\def\cD{\mathcal{D}}
\def\Dm{\cD_m}
\def\rd{\bR^d}

\def\cn{\bC^n}
\def\zd{\bZ^d}
\def\rdd{\bR^{2d}}
\def\zdd{\bZ^{2d}}
\def\la{\langle}
\def\ra{\rangle}
\def\lc{\left(}
\def\rc{\right)}
\def\lV{\left\lVert}
\def\rV{\right\rVert}
\def\lv{\left\lvert}
\def\rv{\right\rvert}
\def\supp{\mathrm{supp}}
\def\wt{\widetilde}
\def\*b{*_{\bullet}}
\def\w{\mathrm{w}}
\def\tr{\mathrm{Tr}}

\def\boxy{\square}
\def\S0{S^0_{0,0}}

\def\Bd'{B_{\delta'}}

\def\cBd'{\bar{B}_{\delta'}}
\def\mcn{\bC^{n\times n}}

\begin{document}
	
\title[Time-frequency analysis of the Dirac equation]{Time-frequency analysis of the Dirac equation}
\author{S. Ivan Trapasso}
\address{Dipartimento di Scienze Matematiche, Politecnico di Torino, corso Duca degli Abruzzi 24, 10129 Torino, Italy}
\email{salvatore.trapasso@polito.it}
\subjclass[2010]{35Q41, 42B35, 47G30.}
\keywords{Dirac equation, modulation spaces, Wiener amalgam spaces, pseudodifferential operators, vector-valued time-frequency analysis.}

\begin{abstract}
	The purpose of this paper is to investigate several issues concerning the Dirac equation from a time-frequency analysis perspective. More precisely, we provide estimates in weighted modulation and Wiener amalgam spaces for the solutions of the Dirac equation with rough potentials. We focus in particular on bounded perturbations, arising as the Weyl quantization of suitable time-dependent symbols, as well as on quadratic and sub-quadratic non-smooth functions, hence generalizing the results in \cite{kn}. We then prove local well-posedness on the same function spaces for the nonlinear Dirac equation with a general nonlinearity, including power-type terms and the Thirring model. For this study we adopt the unifying framework of vector-valued time-frequency analysis \cite{wahlberg}; most of the preliminary results are stated under general assumptions and hence they may be of independent interest.
\end{abstract}
\maketitle
\section{Introduction}
In this note we consider the Cauchy problem for the $n$-dimensional Dirac equation with a potential $V$:
\begin{equation}\label{dirac eq}
\begin{cases}
i\partial_{t}\psi(t,x)=(\Dm+V)\psi(t,x),\\
\psi(0,x)=\psi_{0}(x), 
\end{cases} \qquad (t,x)\in\mathbb{R}\times\mathbb{R}^{d}.
\end{equation}
Here $\psi(t,x)=(\psi_1(t,x),\ldots,\psi_n(t,x))\in\mathbb{C}^{n}$
is a vector-valued complex wavefunction and the Dirac operator $\Dm$ is defined by
\begin{equation}
\Dm= 2\pi m\alpha_0 -i\sum_{j=1}^d \alpha_j \partial_j,\label{freediracop}
\end{equation}
where $m\ge0$  (mass) and $\alpha_{0},\alpha_{1},\ldots,\alpha_{d}\in \mcn$ is a set of \textit{Dirac matrices}, i.e. $n\times n$ Hermitian matrices satisfying the identities
\begin{equation}\label{dirac matrix def}
\alpha_{i}\alpha_{j}+\alpha_{j}\alpha_{i}=2\delta_{ij}I_{n},\quad\forall \, 0 \le i,j \le d,
\end{equation}
($I_{n}$ is the $n\times n$ identity matrix). For $d=3$ and $n=4$ the standard choice for such matrices is the so-called Dirac's representation:
\begin{equation}
\alpha_{i}=\left(\begin{array}{cc}
0 & \sigma_{i}\\
\sigma_{i} & 0
\end{array}\right),\quad i=1,2,3,\qquad\alpha_0=\left(\begin{array}{cc}
I_{2} & 0\\
0 & -I_{2}
\end{array}\right),
\end{equation}
where we introduced the Pauli matrices
\begin{equation}
\sigma_{1}=\left(\begin{array}{cc}
0 & 1\\
1 & 0
\end{array}\right),\qquad\sigma_{2}=\left(\begin{array}{cc}
0 & -i\\
i & 0
\end{array}\right),\qquad\sigma_{3}=\left(\begin{array}{cc}
1 & 0\\
0 & -1
\end{array}\right).
\end{equation}
In general, for any $d$ there exist several iterative schemes to obtain a set of Dirac matrices and in general the dependence of the (even) dimension $n=n(d)$ on $d$ is a consequence of the chosen construction \cite{kalf}. 

The study of the Dirac equation, like other dispersive equations, may certainly take advantage from the techniques of modern harmonic analysis. In the last decades we have witnessed an increasing interest in the application to PDEs of strategies and function spaces arising in time-frequency analysis. Even if it is impossible to offer a comprehensive list of results, we suggest the papers \cite{benyi 1,benyi unimod 2,cheng,cn strichartz 2,cn wave 3,cn pot mod 4,cnr rough 6,kato ki 1,kato ki 3,wang paper,zhang} and the monographs \cite{gro book,wang book} as examples of the manifold aspects one can handle from this perspective. 

The optimal environment for this approach is provided by modulation spaces, which were introduced by Feichtinger in the '80s \cite{fei new segal,fei modulation 83}. In the first instance they can be thought of as Besov spaces with isometric boxes in the frequency domain instead of dyadic annuli. In fact, a much more insightful definition is given in terms of the global decay of the phase-space concentration of a function or a distribution. To be precise, given a temperate distribution $f \in \cS'(\rd)$ and a non-zero Schwartz window function $g \in \cS(\rd)$, the short-time Fourier transform $V_g f$ is defined as 
\[ V_gf(x,\xi) = \cF[f g(\cdot - x) ](\xi), \qquad (x,\xi) \in \rdd, \]
where $\cF$ denotes the Fourier transform. The modulation space $M^{p,q}(\rd)$, $1 \le p,q \le \infty$ is the space of  distributions $f \in \cS'(\rd)$ such that $\lV V_g f(x,\xi) \rV_{L^q(\rd_{\xi}; L^p(\rd_x))} < \infty$. A better control on the regularity is achieved by introducing weights of polynomial type: for $r,s \in \bR$ the $M^{p,q}_{r,s}(\rd)$-norm of $f$ is given by $\lV V_g f(x,\xi) \rV_{L^q_s(\rd_{\xi}; L^p_r(\rd_x))}$, where 
\begin{equation} \label{weighted leb sp} u \in L^q_s (\rd)  \Leftrightarrow (1+|\cdot|^2)^{s/2}u \in L^q(\rd), \end{equation}
and similarly for $L^p_r(\rd)$. In particular, the parameter $s \ge 0$ can be interpreted as the degree of fractional differentiability of $f \in M^{p,q}_{r,s}$. A strictly related family of spaces is obtained by reversing the order of integration in the mixed-Lebesgue norm. The space $W^{p,q}_{r,s}(\rd)$, traditionally called Wiener amalgam space, contains distributions $f\in \cS'(\rd)$ satisfying $\lV V_g f(x,\xi) \rV_{L^q_s(\rd_{x}; L^p_r(\rd_{\xi}))} < \infty$. There is in fact a deeper connection among these spaces, since it turns out that the elements of $W^{p,q}_{r,s}(\rd)$ are Fourier transforms of functions in $M^{p,q}_{r,s}(\rd)$; see Section 3 for more details. 

The relevance of these function spaces to the study of dispersive PDEs is closely related to the evolution of the phase-space concentration under the corresponding propagators. As an example, while the Schr\"odinger propagator $e^{it\triangle}$ is not bounded on $L^p(\rd)$ except for $p=2$, it is a bounded unimodular Fourier multiplier on any modulation space $M^{p,q}(\rd)$ \cite{benyi unimod 2}. Many results of this type, including improved dispersive and Strichartz estimates, are also known for the wave equation and the Klein-Gordon equation (see the list of papers above).

To the best of our knowledge, the only contribution in this spirit concerning the Dirac equation is the recent paper \cite{kn} by Kato and Naumkin. The authors proved estimates for the solutions of the Dirac equation \eqref{dirac eq} in the free case (Theorem 1.1) and also for quadratic and subquadratic time-dependent smooth potentials (Theorem 1.2); the latter setting also includes an electromagnetic potential with linear growth. Broadly speaking, the main difficulty in dealing with \eqref{dirac eq} lies in that it is a system of coupled equations, hence a strategy for disentangling the components is needed. For instance, this can be done approximately at the level of phase space (see \cite[Eq. 3.17]{kn}) or by projection onto the spectrum of the Dirac operators (see the proof of the dispersive estimate \cite[Eq. 1.8]{kn}). Another standard procedure consists of exploiting the connection with the wave and Klein-Gordon equations when $m=0$ and $m>0$ respectively. Nevertheless, when a non-zero potential $V$ is taken into account most of these procedures loose their usefulness and new ideas are required (cf. for instance \cite{cacciafesta da 2,cacciafesta f 3,d'ancona,erdogan}).

The first aim of this paper is to offer a different point of view that does not require an explicit decoupling technique nor any preliminary knowledge about the Klein-Gordon equation. A naive look at \eqref{dirac eq} would suggest to treat it like a Schr\"odinger-type equation with matrix-valued Hamiltonian $\cH=\Dm + V$. For the free case ($V=0$) the corresponding propagator $U(t)= e^{-it\Dm}$ can be formally viewed as a Fourier multiplier with matrix symbol 
\begin{equation} \label{symb intro} \mu_t(\xi) = \exp\left[ -2\pi i t \lc m\alpha_0 + \sum_{j=1}^d \alpha_j \xi_j \rc  \right]. \end{equation}
This perspective naturally leads to consider estimates on vector-valued modulation and Wiener amalgam spaces by studying the regularity of $\mu_t$ and extending the ordinary boundedness results for Fourier multipliers and more general pseudodifferential operators. Roughly speaking, the definition of the modulation space $M^{p,q}(\rd,E)$, $E$ being a complex Banach space in general, coincides with the one given above with $|\cdot|$ replaced by the norm on $E$; such spaces were first considered by Toft \cite{toft cont 2} and then extensively studied by Wahlberg \cite{wahlberg}. The study of the Dirac equation would only require to consider finite-dimensional vector spaces such as $\cn$ and $\mcn$, so that the subtleties connected with infinite-dimensional target spaces are not relevant here and most of the proofs reduce to componentwise computation. Nevertheless, we decided to embrace this wider perspective and thus the first part of the paper is devoted to extend some results of scalar-valued time-frequency analysis to the vector-valued context. In our opinion, the price of developing these tools in full generality is repaid by a unifying and powerful framework which provides very natural and compact proofs for the main results on the Dirac equation. In passing, we remark that the core of results concerning vector-valued time-frequency analysis is in fact of independent interest and falls into the larger area of infinite-dimensional harmonic analysis \cite{girardi class 1,hytonen 2,weis}, with possible applications to abstract evolution equations \cite{amann quasilin 1,arendt} and generalized stochastic processes \cite{fei hor}.

In that spirit, we are then able to prove the following estimates for the free Dirac propagator. 

\begin{theorem}\label{maint free}
	Let $1\le p,q \le \infty$ and $r,s \in \bR$; denote by $X$ any of the spaces $M^{p,q}_{r,s}(\rd,\cn)$ or $W^{p,q}_{r,s}(\rd,\cn)$. Let $\psi(t,x)$ be the solution of \eqref{dirac eq} with $V\equiv 0$. For any $t\in\bR$ there exists a constant $C_X(t)>0$ such that
	\[ 
	\lV \psi(t,\cdot) \rV_{X} \le C_X(t) \lV\psi_0 \rV_{X}. 
	\] 
	In particular, if $X=M^{p,q}_{0,s}(\rd,\cn)$ there exists a constant $C'>0$ such that
	\begin{equation}\label{C(t) estimate free}
	C_X(t) \le C' (1+|t|)^{d|1/2-1/p|}.
	\end{equation}
\end{theorem}
While the results are not unexpected in themselves in view of the discussion above on the connection with the Klein-Gordon propagator, we remark that our method improves the known estimates in two aspects. First, we are able to cover weighted modulation and Wiener amalgam spaces with no extra effort, resulting in a more precise description of the action of the propagator (no loss of derivatives in Theorem \ref{maint free} or asymptotic smoothing in Theorem \ref{smoothing est} below). On the other hand, at least for modulation spaces we are able to explicitly characterize the time-dependence of the constant $C(t)$ in \eqref{C(t) estimate free} in a straightforward way, essentially by inspecting the symbol \eqref{symb intro}.

The second purpose of this note is to provide boundedness results on modulation and Wiener amalgam spaces for suitable potentials $V$ in \eqref{dirac eq}. We relax the regularity assumptions in \cite{kn} in two aspects. First, we replace the multiplication operator by $V$ with a genuine matrix pseudodifferential operator $\sigma^{\w}$ in the Weyl form, where the matrix symbol $\sigma$ belongs to the so-called Sj\"ostrand class \cite{sjo}. In the ordinary scalar-valued framework this is a prime example of an exotic symbol class still yielding bounded Weyl operators on $L^2(\rd)$. A closer inspection reveals that this function space is nothing but the modulation space $M^{\infty,1}$ and it is well known that symbols in this space associate with bounded operators on any (unweighted) modulation space \cite{gro sjo}. This characterization also extends to operator-valued symbols on Hilbert-valued modulation spaces \cite{wahlberg}. 
In addition, while the dependence on the time of the potential $V$ is assumed to be smooth in \cite{kn}, we require here a milder condition, namely continuity for the narrow convergence; see Definition \ref{nar conv} for a precise characterization. In the following claim we use the spaces $\cM^{p,q}_{r,s}$ and $\cW^{p,q}_{r,s}$ defined as the closure of the Schwartz class in the corresponding modulation and Wiener amalgam spaces respectively.

\begin{theorem} \label{maint pert}
Let $1\le p,q \le \infty$, $\gamma \ge0$ and $r,s \in \bR$ be such that $|r|+|s|\le \gamma$; denote by $X$ any of the spaces $\cM^{p,q}_{r,s}(\rd,\cn)$ or $\cW^{p,q}_{r,s}(\rd,\cn)$. Let $T>0$ be fixed and assume the map $[0,T]\ni t \mapsto \sigma(t,\cdot) \in M^{\infty,1}_{0,2\gamma}(\rd,\mcn)$, to be continuous for the narrow convergence. For any $\psi_0 \in X$ there exists a unique solution $\psi \in C([0,T],X)$ to \eqref{dirac eq} with $V=\sigma(t,\cdot)^\w$. The corresponding propagator is bounded on $x$.
\end{theorem}
%
We then consider the case of potentials with quadratic and sub-quadratic growth as in \cite{kn}. As a consequence of a useful splitting lemma, namely Proposition \ref{sj decomp} below, we are able to prove a generalized rough counterpart of the smooth scenario considered in \cite[Thm. 1.2]{kn}. In particular, the potential contains non-smooth functions with a certain number of derivatives in the Sj\"ostrand class plus a perturbation in the Weyl form. 

\begin{theorem} \label{maint split mod}
	Let $1\le p \le \infty$ and $\psi_0 \in \cM^p(\cn)$. Consider the Cauchy problem \eqref{dirac eq} with potential 
	\begin{equation}\label{V maint split} V=Q I_n+L+\sigma^{\w},\end{equation}
	where \begin{itemize}
		\item $Q:\rd \to \bC$ is such that $\partial^{\alpha}Q \in M^{\infty,1}(\rd)$ for $\alpha \in \bN^d$, $|\alpha|=2$,
		\item $L:\rd \to \mcn $ is such that $\partial^{\alpha}L \in M^{\infty,1}(\rd,\mcn)$ for $\alpha \in \bN^d$, $|\alpha|=1$, and		
		\item $\sigma \in M^{\infty,1}(\rd,\mcn)$.

	\end{itemize}
	 For any $t\in\bR$ there exists a constant $C(t)>0$ such that the solution $\psi$ of \eqref{dirac eq} satisfies 
	\[ 
	\lV \psi(t,\cdot) \rV_{\cM^p} \le C(t) \lV\psi_0 \rV_{\cM^p}. 
	\] 
	Furthermore, if $V$ is as in \eqref{V maint split} and $Q=0$, then for any $1\le p,q\le \infty$ and $t\in\bR$ there exists a constant $C(t)>0$ such that the solution $\psi$ of \eqref{dirac eq} satisfies
	\[ 
	\lV \psi(t,\cdot) \rV_{\cM^{p,q}} \le C(t) \lV\psi_0 \rV_{\cM^{p,q}}. 
	\] 
\end{theorem}

In the last part of the paper we study the local well-posedness for the nonlinear setting, namely
\begin{equation}\label{dirac eq nonlin}
\begin{cases}
i\partial_{t}\psi\left(t,x\right)=\Dm\psi(t,x)+ F(\psi(t,x)),\\
\psi\left(0,x\right)=\psi_{0}\left(x\right), 
\end{cases} \qquad \left(t,x\right)\in\mathbb{R}\times\mathbb{R}^{d},
\end{equation}
where the nonlinear term $F$ considered below comes in the form of a vector-valued real-analytic entire function $F:\cn \to \cn$ such that $F(0)=0$, i.e. 
\begin{equation} \label{gen nonlin} F_j(z)= \sum_{\alpha,\beta \in \bN^n} c^j_{\alpha,\beta} z^\alpha \bar{z}^\beta, \qquad j=1,\ldots,n. \end{equation} We remark that this general choice includes nonlinearities of power type, such as
\begin{equation}\label{power nonlin} F(\psi)= |\psi|^{2k}\psi,\qquad k \in \bN; \end{equation}
and the cubic nonlinearity known as the Thirring model, namely
\begin{equation}\label{thirring mod} F(\psi)= (\alpha_0\psi,\psi) \alpha_0 \psi; \end{equation}
The choice of even powers in \eqref{power nonlin} and entire functions as in \eqref{gen nonlin} are standard in the context of modulation and amalgam spaces, because of the Banach algebra property enjoyed by certain spaces of these families \cite{sugi nonlin}. On the other hand, the nonlinear spinor field appearing in the Thirring model has been largely investigated; cf. for instance  \cite{bejenaru,huh 1,machihara 1,naumkin}, also in view of its physical relevance - it is a model for self-interacting Dirac fermions in quantum field theory \cite{soler,thirring}. 

The main result in this respect reads as follows.
\begin{theorem}\label{maint nonlin}
	Let $1\le p \le \infty$ and $r,s\ge0$; denote by $X$ any of the spaces $M^{p,1}_{0,s}(\rd,\cn)$ or $W^{1,p}_{r,s}(\rd,\cn)$. If $\psi_0 \in X$ then there exists $T = T(\lV \psi_0 \rV_X)$ such that the Cauchy problem \eqref{dirac eq nonlin} with $F$ as in \eqref{gen nonlin} has a unique solution $\psi \in C^0([0,T],X)$. 
\end{theorem}
We conclude this introduction by emphasizing a few aspects that may be further developed in the context of modulation spaces, such as Strichartz estimates and perturbations due to a  magnetic field, i.e. the Dirac operator in \eqref{freediracop} becomes $\cD_{m,A} = 2\pi m \alpha_0 - i \sum_{j=1}^{d} \alpha_j (\partial_j - iA_j)$, where $A(x)=(A_1(x),\ldots,A_d(x))$, $x \in \rd$, is a static magnetic potential. We also point out that more general nonlinear terms could be considered, for instance as in the Soler model \cite{soler} and other interactions arising in condensed matter; cf.  \cite{pelinovsky} for the state of the art in 1+1 dimensions. 

\section{Preliminaries}

\subsection{Notation} We define $t^2=t\cdot t$, for $t\in\mathbb{R}^d$, and
$x\cdot y$ is the scalar product on $\mathbb{R}^{d}$. The Schwartz class is denoted by  $\mathcal{S}(\mathbb{R}^{d})$, the space of temperate distributions by  $\mathcal{S}'(\mathbb{R}^{d})$. The brackets  $\langle  f,g\rangle $ denote the extension to $\cS' (\mathbb{R}^{d})\times \cS (\mathbb{R}^{d})$ of the inner product $\langle f,g\rangle=\int f(t){\overline {g(t)}}dt$ on $L^2(\mathbb{R}^{d})$. 

The characteristic function on a set $A\subseteq E$ is denoted with $\chi_{A}$. For $x=(x_1,\ldots,x_d) \in \rd$ we set $|x|_{\infty} = \max\{|x_1|,\ldots,|x_d|\}$. 

The conjugate exponent $p'$ of $p \in [1,\infty]$ is defined by $1/p+1/p'=1$. The symbol $\lesssim$ means that the underlying inequality holds up to a positive constant factor $C>0$:
$$ f\lesssim g\quad\Rightarrow\quad\exists C>0\,:\,f\le Cg. $$ We write $f\asymp g$ to say that both $f \lesssim g$ and $g\lesssim f$ hold. 

We choose the following normalization for the Fourier transform:
\[
\mathcal{F}f\left(\xi\right)=\hat{f}(\xi)=\int_{\mathbb{R}^{d}}e^{-2\pi ix\cdot \xi}f\left(x\right) d x,\qquad\xi\in \rd.
\] 
We define the involution $^{*}$ as $f^*(t)=\overline{f(-t)}$. 
For any $x,\xi \in \mathbb{R}^{d}$, the modulation $M_{\xi}$ and translation $T_{x}$ operators are defined as 
\[
M_{\xi}f\left(t\right)= e^{2\pi it \cdot \xi}f\left(t\right),\qquad T_{x}f\left(t\right)= f\left(t-x\right).
\]
For $m>0$ and $t \in \rd$ we set $\la \xi \ra_m \coloneqq \sqrt{m^2 + \xi^2}$. We omit the subscript for $m=1$, namely $\la \xi \ra$ stands for $\la \xi \ra_1$.
Denote by $J$ the canonical symplectic matrix in $\mathbb{R}^{2d}$:
\[
J=\left(\begin{array}{cc}
0_d & I_d\\
-I_d & 0_d
\end{array}\right).
\] \\  
In what follows we always denote by $E$ a complex Banach space with norm $|\cdot|_E$, whereas the symbol $H$ is reserved for a complex separable Hilbert space. The topological dual space of $E$ is denoted by $E'$. The brackets $(\cdot,\cdot)$ are used for the duality between $E'$ and $E$ and in particular for the inner product in $H$ - we assume $(\cdot,\cdot)$ to be conjugate-linear in the second argument. Given two normed spaces $X$ and $Y$, the space of continuous linear operators $X\to Y$ with the topology of bounded convergence is denoted by $\cL(X,Y)$, whereas we write $\cL_s(X,Y)$ for the same set endowed with the strong operator topology. The space of trace-class operators on $H$ is denoted by $\cL^1(H)$. \\ 
The space of smooth $E$-valued functions with bounded derivatives of any order larger than $k \in \bN$ is 
\[ C^{\infty}_{\ge k}(\rd,E) \coloneqq \left\{f \in C^{\infty}(\rd,E) \,:\, |\partial^{\alpha}f| \le C_{\alpha} \quad \forall \alpha\in \bN^d, |\alpha|\ge k \right\}. \] 
Notice that $C^{\infty}_{\ge 0}(\rd)$ coincides with the well-known H\"ormander class $\S0(\rd)$ \cite{gro rze,horm 3}. 
We will occasionally make use of the Dirac notation for projection operators: given $\phi,\psi \in H$, we define 
\[ |\psi\ra \la \phi| : H \to H,\quad  |\psi\ra \la \phi|(w) = (w,\phi)_H \psi.\]\\ 
Given a triple $E_1$, $E_2$ and $E_3$ of complex Banach spaces, we say that the map \[ \bullet : E_1 \times E_2 \to E_3, \quad (x_1,x_2) \mapsto x_3= x_1\bullet x_2 \] is a \textit{multiplication} \cite{amann book 3} if it is a continuous bilinear operator such that $\lV \bullet \rV_{\cL(E_1\times E_2,E_3)} \le 1$. The following are common examples of multiplications that will be used below: 
\begin{enumerate}
	\item multiplication with scalars: $\bC \times E \to E$, $(\lambda,x) \mapsto \lambda x$;
	\item the duality pairing: $E' \times E \to \bC$, $(u,x) \mapsto u(x)$;
	\item the evaluation map: $\cL(E_1,E_2) \times E_1 \to E_2$, $(T,x) \mapsto Tx$;
	\item multiplication in a Banach algebra.
\end{enumerate}

Although neither the concrete expressions of the Dirac matrices nor deep aspects related to the Clifford algebra representation theory are relevant for our purposes, we point out that the conditions \eqref{dirac matrix def} force $n$ to be even and we may assume without loss of generality that
\[ \alpha_0 = \left(\begin{array}{cc} I_{n/2} & 0\\ 0 & -I_{n/2} \end{array}\right).  \]
We refer the interested reader to \cite{kalf,ozawa} for further details.
 
\subsection{Vector-valued function spaces and operators}
The notation and the basic results of analysis on infinite-dimensional spaces are rather standard \cite{amann book 3,girardi class 1,hytonen book 1} and we will not linger over the subtleties arising from the infinite-dimensional context. For the convenience of the reader we briefly collect the main facts of harmonic analysis in the vector-valued context. In what follows we consider functions $f:\rd \rightarrow E$, where $\rd$ is provided with the Lebesgue measure $\mu_L$. \\

The family of \textit{Lebesgue-Bochner spaces} is the natural analogue of Lebesgue spaces of scalar-valued functions. When there is no risk of confusion, we will write $L^p_s(E)$ for $L^p_s(\rd,E)$ and $L^p(E)$ when $s=0$. Notice that $f=(f_1,\ldots,f_n) \in L^p_s(\rd,\cn)$ if and only if $f_j \in L^p_s(\rd)$ for any $j=1,\ldots,n$. 
Most of the usual properties from the scalar-valued case extend in a natural way (with the remarkable exception of duality \cite{hytonen book 1}). 
\begin{proposition}[{\cite{amann book 3,hytonen book 1}}] $(i)$ For any $1\le p \le \infty$, $L^p(\rd,E)$ is a Banach space with the norm $\lV f \rV_{L^p(\rd,E)} = \lV \lv f(\cdot) \rv_{E} \rV_{L^p}$. \\ 
	$(ii)$ $L^2(\rd,H)$ is a Hilbert space with inner product given by \[ \la f,g \ra_{L^2(H)} = \int_{\rd} ( f(t),g(t) )_H dt. \] 
	$(iii)$ (H\"older inequality) Given a multiplication $\bullet : E_1 \times E_2 \to E_3$, $s_1,s_2 \in \bR$ and $1\le p_1,p_2,p \le \infty$ such that $1/p_1+1/p_2=1/p$, if $f\in L^{p_1}_{s_1}(\rd,E_1)$ and $g\in L^{p_2}_{s_2}(\rd,E_2)$ then $f\bullet g \in L^p_{s_1+s_2}(\rd,E_3)$ and $ \lV f \bullet g \rV_{L^p_{s_1+s_2}(E_3)} \le \lV f \rV_{L^{p_1}_{s_1}(E_1)} \lV g \rV_{L^{p_2}_{s_2}(E_2)}$. 
\end{proposition}

\noindent
\textbf{Distributions and Fourier transform (\cite{amann book 3,hytonen book 1}).} 
Recall that the Schwartz class of $E$-valued rapidly decreasing functions $\cS(\rd,E)$ is a Fr\'echet space with the topology induced by the family of seminorms $\{p_{m,E}\}_{m\in \bN}$, where \[ p_{m,E}(f)\coloneqq \sup_{\substack{t\in\rd \\ |\alpha|+|\beta|<m}} \lv t^{\alpha} \partial^{\beta}f(t)\rv_E < \infty, \] and is a dense subset of $L^p(\rd,E)$ for any $1\le p < \infty$. 

The space of $E$-valued temperate distributions $\cS'(\rd,E)$ consists of bounded (conjugate-)linear maps from $\cS(\rd)$ to $E$, that is $\cS'(\rd,E)= \cL(\cS(\rd),E)$. 

For $1\le p \le \infty$ any $p$-integrable $E$-valued function $f$ can be identified with a $E$-valued temperate distribution as usual: 
\[ \la f,g \ra = \int_{\rd} f(t)\overline{g(t)}dt, \qquad g\in \cS(\rd). \]  Notice that this is a further meaning for the brackets $\la \cdot,\cdot \ra$.

The Fourier transform can be initially defined as a Bochner integral for $f \in L^1(\rd,E)$ and its restriction to $\cS(\rd,E)$ yields a continuous automorphism that enjoys the usual properties (e.g., the Riemann-Lebesgue lemma, the inversion theorem, the relations with translation, modulation and differentiation). There is a notable exception: while $\cF:L^1(\rd,E) \to L^{\infty}(\rd,E)$, the Hausdorff-Young inequality does not hold in general \cite{hytonen book 1}. In particular, it is a deep result by Kwapie\'n \cite{kwap} that the Parseval-Plancherel theorem yields the extension of $\cF$ to a unitary operator on $L^2(\rd,E)$ if and only if $E$ is isomorphic to a Hilbert space.

Nevertheless, the Fourier transform extends to an isomorphism on $\cS'(\rd,E)$ as follows: 
\[ \la \hat{f},\hat{g} \ra \equiv \la f, g \ra, \qquad f \in \cS'(\rd,E),\,g\in \cS(\rd). \] 

For future convenience we define the (Bochner-)Fourier-Lebesgue spaces $\cF L^q_s(\rd,E)$ consisting of distributions $f\in \cS'(\rd,E)$ such that \[ \lV f\rV_{\cF L^q_s(E)} \coloneqq \lV \cF^{-1}f \rV_{L^q_s(E)} < \infty. \]

The following Bernstein-type lemma can be proved just as in the scalar-valued case; cf. \cite[Prop. 1.11]{wang book}.

\begin{lemma}\label{bernstein lem}
	Let $N > d/2$ be an integer and $\partial^k_j f \in L^2(\rd,H)$ for any $j=1,\ldots,d$ and $0 \le k \le N$. Then 
	\begin{equation}\label{bernstein est}
	\lV f \rV_{\cF L^1(H)} \lesssim \lV f \rV_{L^2(H)}^{1-d/2N} \lc \sum_{j=1}^d \lV \partial^N_j f \rV_{L^2(H)} \rc^{d/2N}.
	\end{equation}
\end{lemma}
 
\noindent
\textbf{Convolution and Fourier multipliers.} The convolution of vector-valued functions can be meaningfully defined as soon as the target spaces are provided with a multiplication structure \cite{amann book 3,hytonen book 1}. 
The convolution of $f \in \cS'(\rd,E)$ with a Schwartz function $g \in \cS(\rd)$ is the distribution $f*g \in \cS'(\rd,E)$ such that
\[ \la f*g, \phi\ra \equiv \la f, g^* * \phi\ra, \qquad \forall \phi \in \cS(\rd). \]
In fact, $f*g \in C^{\infty}(\rd,E)$ is a function of polynomial growth together with all its derivatives. Moreover, for $f \in L^p(\rd,E)$ and $g \in L^1(\rd)$ we recover the ordinary convolution
\[ f*g(x) = \int_{\rd} f(x-y)g(y)dy, \]
which is a well-defined Bochner integral for a.e. $x \in \rd$. In particular, $f*g \in L^p(\rd,E)$ with $\lV f*g \rV_{L^p(E)} \le \lV f \rV_{L^p(E)} \lV g \rV_{L^1}$. 
The $\bullet$-convolution $f_1 *_{\bullet} f_2$ of $f_1 \in \cS(\rd,E_1)$ and $f_2 \in \cS'(\rd,E_2)$ can be similarly defined as a smooth $E_3$-valued function for any multiplication $\bullet: E_1 \times E_2 \to E_3$ \cite[Thm. 1.9.1]{amann book 3}.
We state some results that will be needed below. The proofs of more general versions of these facts can be found in \cite[Sec. 1.9]{amann book 3}. See also \cite{kerman}. 

\begin{proposition}\label{convol properties} \begin{enumerate}[label=(\roman*)] 
	\item (Young inequality) Let $1 \le p,q,r \le \infty$ satisfy $1/p + 1/q = 1+ 1/r$ and $s_1,s_2,s_3 \in \bR$ satisfy 
	\[ s_1+s_3 \ge 0, \quad s_2+s_3 \ge 0, \quad s_1+s_2 \ge 0. \] If $f\in L^p_{s_1}(\rd,E_1)$ and $g \in L^q_{s_2}(\rd,E_2)$, then $f*_{\bullet}g \in L^r_{-s_3}(\rd,E_3)$, with 
	\[ \lV f*_{\bullet}g \rV_{L^r_{-s_3}(E_3)} \lesssim \lV f \rV_{L^p_{s_1}(E_1)} \lV g \rV_{L^q_{s_2}(E_2)}. \]
	\item For any $f \in \cS'(\rd,E_1)$ and $g \in \cS(\rd,E_2)$:
	\[ \cF(f *_{\bullet} g) = \hat{f} \bullet \hat{g}. \]
\end{enumerate}
\end{proposition}	


We then introduce the \textit{Fourier multiplier} with symbol $\mu \in \cS'(\rd,E_1)$ as the linear map  
\[ \mu(D)f \coloneqq \cF^{-1}(\mu \bullet \hat f) = \cF^{-1}\mu *_{\bullet} f \in \cS'(\rd,E_3), \] the domain consisting of all $f \in \cS'(\rd,E_2)$ such that the latter convolution is well defined \cite{amann book 3}. 

\subsection{Vector-valued time-frequency analysis} The short-time Fourier transform of a vector-valued distribution $f\in \cS'(\rd,E)$ with respect to a non-zero window function $g \in \cS(\mathbb{R}^{d})$ is defined \cite{wahlberg} as the distribution 
\begin{equation} \label{STFTdef}
V_gf(x,\xi)  \coloneqq \langle f,M_\xi T_x g\rangle.
\end{equation} 
Equivalent representations of $V_gf$ are the following ones, whenever meaningful (assume for instance $f\in L^2(\rd,E)$):
\begin{align}
V_gf(x,\xi) & =\int_{\mathbb{R}^{d}} e^{- 2\pi iy \xi }\,f(y)\, {\overline {g(y-x)}} \,dy 
\\  & =\mathcal{F} (f\cdot \overline{T_x g})(\xi) \\ & = e^{-2 \pi i x \cdot\xi} (f*M_{\xi}g^*)(x) \label{stft convol} \\ & = \la \hat{f}, T_{\xi}M_{-x}\hat{g} \ra \\ & = e^{2\pi i x\cdot \xi}V_{\hat{g}}\hat{f}(\xi,-x) \label{fund id tfa}.
\end{align}
It can be proved \cite[Lem. 2.1]{wahlberg} that $V_gf \in C^{\infty}(\rdd,E)$ and \[
\lv V_gf(x,\xi) \rv_E \leq C(1+ \lvert x \rvert + \lvert \xi \rvert )^N,
\] for some $C>0$, $N\in \bN$ and any $x,\xi \in \rd$. 

\begin{definition}
	Let $1 \le p,q \le \infty$ and $r,s\in \bR$. The $E$-valued modulation space $M^{p,q}_{r,s}(\rd,E)$ consists of distributions $f \in \cS'(\rd,E)$ such that
	\begin{equation}\label{mod sp norm}
	\lVert f \rVert_{M^{p,q}_{r,s}} =  \left( \int_{\rd} \left( \int_{\rd} \lv V_gf(x,\xi) \rv_E^p \la x \ra^{rp} dx \right)^{q/p} \la \xi \ra^{sq} d \xi \right)^{1/q} < \infty, \end{equation} for some $g \in \cS(\rd)$, with suitable modification for $p=\infty$ or $q=\infty$.
\end{definition}


If $r=s=0$ we omit the indices and write $M^{p,q}$. Furthermore, we write $M^p$ for $M^{p,p}$ and $M^{p,q}(E)$ for $M^{p,q}(\rd,E)$ when there is no risk of confusion. We remark that more general weights may be taken into account \cite{wahlberg}.

Most of the ordinary theory extends to the vector-valued context by simply substituting $| \cdot |$ with $|\cdot|_E$ in the proofs. For our purposes, it is enough to mention the following properties. 

\begin{proposition} \label{mod sp prop} Let $1\le p,q \le \infty$ and $r,s\in \bR$. \begin{enumerate}[label=(\roman*)]
		\item $M^{p,q}_{r,s}(\rd,E)$ is a Banach space with the norm \eqref{mod sp norm}, which is independent of the window function $g$ (i.e., different windows yield equivalent norms).
		\item If $p,q < \infty$ the Schwartz class $\cS(\rd,E)$ is dense in $M^{p,q}_{r,s}(\rd,E)$. 
		\item If $p_1\le p_2$, $q_1 \le q_2$ and $r_2\le r_1$, $s_2\le s_1$, then $M^{p_1,q_1}_{r_1,s_1}(\rd,E) \hookrightarrow M^{p_2,q_2}_{r_2,s_2}(\rd,E)$. 
		\item If $E= \bC^{a\times b}$, then $f \in M^{p,q}_{r,s}(\rd,\bC^{a \times b})$ if and only if $f_{ij} \in M^{p,q}_{r,s}(\rd,\bC)$ for any $i=1,\ldots,a$, $j=1,\ldots,b$. 
	\end{enumerate}
\end{proposition}

\begin{remark} \label{rem dual mod sp} In contrast to the aforementioned properties, duality is a quite subtle question (cf. \cite{wahlberg}). In order to avoid related issues, which usually occur when $p,q \in \{ 1,\infty\}$, it is convenient to introduce the space $\cM^{p,q}_{r,s}(\rd,E)$, namely the closure of $\cS(\rd,E)$ with respect to the $M^{p,q}_{r,s}$ norm. In particular we have $\cM^{p,q}_{r,s}(\rd,E) = M^{p,q}_{r,s}(\rd,E)$ for $1\le p,q < \infty$. 
\end{remark}

By reversing the order of integrals in the definition of modulation spaces one obtains a new family of spaces. 
\begin{definition}
	Let $1 \le p,q \le \infty$ and $r,s \in \bR$. The $E$-valued modulation space $W^{p,q}_{r,s}(\rd,E)$ consists of distributions $f \in \cS'(\rd,E)$ such that
	\[
	\lVert f \rVert_{W^{p,q}_{r,s}} = \left( \int_{\rd} \left( \int_{\rd} \lv V_gf(x,\xi) \rv_E^p \la \xi \ra^{rp} d\xi \right)^{q/p} \la x \ra^{sq} dx \right)^{1/q} < \infty, \] for some $g \in \cS(\rd)$, with suitable modification for $p=\infty$ or $q=\infty$.
\end{definition}

From \eqref{fund id tfa} we immediately get $\lV \hat{f} \rV_{M^{p,q}_{r,s}} = \lV f \rV_{W^{p,q}_{r,s}}$, that is $\cF M^{p,q}_{r,s}(\rd,E) = W^{p,q}_{r,s}(\rd,E)$. This should not come as a surprise, since Feichtinger originally designed modulation spaces as Wiener amalgam spaces on the Fourier side \cite{fei modulation 83,fei looking}. Furthermore, the results stated in Proposition \ref{mod sp prop} have an identical counterpart for Wiener amalgam spaces, it is enough to replace $M^{p,q}_{r,s}$ with $W^{p,q}_{r,s}$ in the claim.  

As already noted by Wahlberg \cite{wahlberg}, the spaces $W^{p,q}_{r,s}(\rd,E)$ are in fact Wiener amalgam spaces in the broadest sense, namely 
\[ W^{p,q}_{r,s}(\rd,E) = W(\cF L^p_r(\rd,E), L^q_s(\rd)), \] 
hence they inherit certain properties from their local and global components \cite{fei wien 83}. In order to exploit this connection we introduce a useful equivalent discrete norm for amalgam spaces. Recall that a bounded uniform partition of function (BUPU)  $(\{\psi_i\}_{i\in I}, (x_i)_{i \in I}, U)$ consists of a family of non-negative functions in $\cF L^1_{|r|}(\rd)$  $\{ \psi_i \}_{i\in I}$ such that the following conditions are satisfied:
\begin{enumerate}
	\item $\sum_{i \in I} \psi_i (x) =1$, for any $x\in \rd$;
	\item $\sup_{i\in I} \lV \psi_i \rV_{\cF L^1_{|r|}} < \infty$;
	\item there exist a discrete family $(x_i)_{i\in I}$ in $\rd$ and a relatively compact set $U\subset \rd$ such that $\supp(\psi_i) \subset x_i + U$ for any $i \in I$, and
	\item $\sup_{i\in I} \#\{j\,:\, x_i+U \cap x_j+U \ne \emptyset \} < \infty$.
\end{enumerate}
A general result in the theory of amalgam spaces is the following norm equivalence in the spirit of decomposition spaces \cite{fei wien 83,fei dec 1,fei dec 2}:
\begin{equation} \label{was disc norm}
\lV f \rV _{W^{p,q}_{r,s}(\rd,E)}  \asymp \lc \sum_{i \in I} \lV f \, \psi_i \rV_{\cF L^p_r(\rd,E)}^q \la x_i \ra^{sq} \rc^{1/q}.
\end{equation} 

A similar characterization holds for modulation spaces \cite{fei modulation 83,zhang}, providing a norm comparable to that of Besov spaces:
\begin{equation} \label{mod disc norm}
\lV f \rV_{M^{p,q}_{r,s}(\rd,E)} \asymp \lc \sum_{i \in I} \lV \boxy_i f \rV_{L^p_r(\rd,E)}^q \la x_i \ra^{sq} \rc^{1/q}
\end{equation}
where we introduced the frequency-uniform decomposition operators 
\[ \boxy_i \coloneqq \cF^{-1}\psi_i \cF, \qquad i\in I. \]

Many properties satisfied by modulation spaces carry over to Wiener amalgam spaces in view of the isomorphism established by the Fourier transform. In particular, a Young type result can be obtained after a suitable modification of the proof of \cite[Thm. 3]{fei wien 83}.
\begin{theorem}\label{was conv} 
Let $\bullet : E_1 \times E_2 \to E_3$ be a multiplication for the triple of Banach spaces $(E_1,E_2,E_3)$. For any $1\le p_1,p_2,p_3,q_1,q_2,q_3 \le \infty$ and $r_1,r_2,r_3,s_1,s_2,s_3 \in \bR$ such that 
	\[ \cF L^{p_1}_{r_1}(\rd,E_1) *_{\bullet} \cF L^{p_2}_{r_2}(\rd,E_2) \hookrightarrow \cF L^{p_3}_{r_3}(\rd,E_3), \]
	\[ L^{q_1}_{s_1}(\rd) * L^{q_2}_{s_2}(\rd) \hookrightarrow L^{q_3}_{s_3}(\rd), \]
	the following inclusion holds:
	\begin{equation} W^{p_1,q_1}_{r_1,s_1}(\rd,E_1) *_{\bullet} W^{p_2,q_2}_{r_2,s_2}(\rd,E_2) \hookrightarrow W^{p_3,q_3}_{r_3,s_3}(\rd,E_3). \end{equation}
\end{theorem} 
\begin{proof} For the benefit of the reader we sketch here a short proof in the spirit of \cite[Thm. 11.8.3]{heil}. We consider as BUPU for $W^{p,q}_{r,s}(\rd,E)$ the family $\{ \psi_k\}_{k\in\zd} \subset \cF L^1_{|r|}(\rd)$ defined by 
	\[\psi_k(t)=\frac{\phi(t-k)}{\sum_{k\in\zd}\phi(t-k)}, \quad t\in\rd, \]
	for a fixed $\phi \in C^{\infty}_c(\rd)$ such that $\phi(t)=1$ for $t\in[0,1]^d$ and $\phi(t)=0$ for $t \in \rd\setminus[-1,2]^d$. After introducing the control functions 
	\[ \Psi_{f,p,r,E}(k) \coloneqq \lV f\, \psi_k \rV_{\cF L^p_r(\rd,E)}, \quad k \in \bZ^d, \] the equivalent norm \eqref{was disc norm} becomes
	\[ \lV f \rV_{W^{p,q}_{r,s}(\rd,E)} \asymp \lc \sum_{k\in\zd} \lV f \, \psi_k \rV_{\cF L^p_r(\rd,E)}^q \la k \ra^{qs} \rc^{1/q} \asymp \lV \Psi_{f,p,r,E} \rV_{\ell^q_s(\zd)}. \]
	For $f \in W^{p_1,q_1}_{r_1,s_1}(\rd,E_1)$ and $g\in W^{p_2,q_2}_{r_2,s_2}(\rd,E_2)$ set $f_m = f \psi_m$, $g_n = g  \psi_n$ for $m,n \in \zd$. In view of the support property \cite[Rem. 1.9.6(f)]{amann book 3} and the properties of BUPUs, we have 
	\[ \supp(f_m \*b g_n) \subset \supp(f_m) + \supp(g_n) =  m+n+2\,\supp \psi. \]
	It is then clear that the cardinality of the set 
	$ J_k \coloneqq \{ (m,n) \in \zdd: \supp((f_m \*b g_n) \psi_k) \neq \emptyset \}$ is finite for any $k \in \zd$ and is uniformly bounded with respect to $m,n,k$. In fact, notice that \[ J_k = \{ (m,n) \in \zdd: m = k-n+\alpha, \, |\alpha| \le N(d) \},\] for a fixed constant $N(d) \in \bN$ depending only on the dimension $d$. Therefore, an easy computation yields \[ \Psi_{f\*b g, p_3,r_3,E_3}(k) = \sum_{|\alpha|\le N(d)} \Psi_{f,p_1,r_1,E_1} * \Psi_{g,p_2,r_2,E_2}(k+\alpha), \] and hence \[ 	\lV f\*b g \rV_{W^{p_3,q_3}_{r_3,s_3}(\rd,E_3)} \lesssim \lV f \rV_{W^{p_1,q_1}_{r_1,s_1}(\rd,E_1)} \lV g \rV_{W^{p_2,q_2}_{r_2,s_2}(\rd,E_2)},  \] that is the claim.
\end{proof}

\begin{remark}\label{minfty mult} In view of the relation with modulation spaces and Young inequality for convolution, under the same assumptions of the previous theorem we also have 
	\begin{equation}\label{pointw mod sp} M^{p_1,q_1}_{r_1,s_1}(\rd,E_1) \bullet M^{p_2,q_2}_{r_2,s_2}(\rd,E_2) \hookrightarrow M^{p_3,q_3}_{r_3,s_3}(\rd,E_3). \end{equation}
\end{remark}

An interesting relation between modulation and Wiener amalgam spaces is given by the following generalized Hausdorff-Young inequality, which is a direct consequence of Minkowski's integral inequality: 
\begin{equation}\label{gen hy} M^{p,q}_{r,s}(\rd,E) \hookrightarrow W^{q,p}_{s,r}(\rd,E), \quad 1 \le q \le p \le \infty, \, r,s\in \bR. \end{equation}

\subsection{Fourier multipliers}\label{fou mult sec}

We now provide sufficient conditions on the symbol of a Fourier multiplier in order for it to be bounded on modulation and Wiener amgalgam spaces. 

\begin{proposition} \label{fou mult mod} 
	Let $\bullet: E_0 \times E_1 \to E_2$ be a multiplication and $\mu \in W^{1,\infty}_{|r|,\delta}(\rd,E_0)$ for some $r,\delta \in \bR$. The Fourier multiplier $\mu(D)$ is bounded from  $M^{p,q}_{r,s}(\rd,E_1)$ to $M^{p,q}_{r,s+\delta}(\rd,E_2)$ for any $1\le p,q \le \infty$ and $s \in \bR$. In particular, 
	\[ 	\lV \mu(D)f \rV_{M^{p,q}_{r,s+\delta}(E_2)} \lesssim \lV \mu \rV_{W^{1,\infty}_{|r|,\delta}(E_0)} \lV f \rV_{M^{p,q}_{r,s}(E_1)}, \quad f \in M^{p,q}_{r,s}(E_1).  \]
\end{proposition}

\begin{proof} 
	The proof is a straightforward generalization of the argument used in the scalar-valued case; see for instance \cite[Lem. 8]{benyi unimod 2}. We remark that Theorem \ref{was conv} and the associativity of $\bullet$-convolutions \cite[Rem. 1.9.6(c)]{amann book 3} are required. \end{proof}

A similar result holds for Fourier multipliers on Wiener amalgam spaces. 
\begin{proposition}\label{fou mult was} Let $\bullet: E_0 \times E_1 \to E_2$ be a multiplication and $\mu \in M^{\infty,1}_{\delta,|s|}(\rd,E_0)$ for some $s,\delta \in \bR$. The Fourier multiplier with symbol $\mu$ is bounded from $W^{p,q}_{r,s}(\rd,E_1)$ to $W^{p,q}_{r+\delta,s}(\rd,E_2)$ for any $1\le p,q \le \infty$ and $r \in \bR$. In particular, 
	\[ \lV \mu(D)f \rV _{W^{p,q}_{r+\delta,s}(E_2)} \lesssim  \lV  \mu \rV _{M^{\infty,1}_{\delta,|s|}(E_0)} \lV f \rV _{W^{p,q}_{r,s}(E_1)}, \quad f \in W^{p,q}_{r,s}(E_1). \]
\end{proposition}
\begin{proof}
	Recall that $W^{p,q}_{r,s}(\rd,E) = W(\cF L^p_r(\rd,E),L^q_s(\rd))$. Theorem \ref{was conv} and the relation $\cF M^{p,q}_{r,s} = W^{p,q}_{r,s} $ thus yield
	\begin{align*}
	\lV \mu(D)f \rV _{W^{p,q}_{r+\delta,s}(E_2)} & = \lV  \cF^{-1}\mu *_{\bullet} f \rV _{W^{p,q}_{r+\delta,s}(E_2)} \\
	& \lesssim \lV  \cF^{-1}\mu \rV _{W^{\infty,1}_{\delta,|s|}(E_0)} \lV f \rV _{W^{p,q}_{r,s}(E_1)} \\
	& \lesssim \lV  \mu \rV _{M^{\infty,1}_{\delta,|s|}(E_0)} \lV f \rV _{W^{p,q}_{r,s}(E_1)}.
	\end{align*} 
\end{proof}

\subsection{The Wigner distribution and the Weyl transform}\label{wig weyl sec}
Given $f,g \in L^2(\rd,H)$, the Wigner distribution $W(f,g)(x,\xi)\in \cL(H)$, $x,\xi \in \rd$,  is defined as follows:  
\begin{equation} \label{wig def}
W(f,g)(x,\xi) = [\cF \mathfrak{T}_s P(f,g)(x,\cdot)](\xi),
\end{equation}
where we introduced the projector-valued function 
\[ 
P(f,g) : \rdd \to \cL^1(H), \qquad P(f,g)(x,y) \coloneqq |f(x)\ra \la g(y)|,
\]
and $\mathfrak{T}_s$ is the linear transformation acting on $F:\rdd \to H$ as
\[ \mathfrak{T}_sF(x,y) = F\lc x+\frac{y}{2}, x- \frac{y}{2} \rc. \]
It is therefore clear that $W(f,g) : \rdd \to \cL^1(H)$ and in particular \cite{folland,wahlberg}
\[ \lc W(f,g)(x,\xi)u,v \rc_H = \int_{\rd} e^{-2\pi i y\cdot \xi}\lc f(x+y/2),v \rc_H \overline{\lc g(x-y/2),u \rc_H} dy,\] for any $ u,v \in H$. More concisely, we have 
\[ \lc W(f,g)(x,\xi)u,v \rc_H = W(\wt{f_v},\wt{g_u})(x,\xi), \]where on the right-hand side we have the ordinary Wigner distribution of the functions
\[ \wt{f_v}(t) = (f(t),v)_H,\quad \wt{g_u}(t) = (g(t),u)_H. \]

The following properties of the Wigner distributions are well known in the standard setting \cite{gro book} and can be easily derived in the vector-valued context.

\begin{proposition} \label{wig prop} For any $f,g\in \cS(\rd,H)$ and $x,\xi \in \rd$: \begin{enumerate}[label=(\roman*)]
	\item $W(f,g) \in \cS(\rdd,\cL^1(H))$.
	\item $W(f,g)(x,\xi) = W(\hat{f},\hat{g})(\xi,-x)$. 
	\item $\int_{\rd} W(f,g)(x,\xi) dx = |\hat{f}(\xi)\ra \la \hat{g}(\xi)|. $
	\item $\int_{\rd} W(f,g)(x,\xi) d\xi = |f(x)\ra \la g(x)|. $
\end{enumerate}	
\end{proposition}

The Wigner transform can be extended to $f,g\in \cS'(\rd,H)$ as follows \cite{wahlberg}. Let $\Phi= W(\phi_1, \phi_2)$ for $\phi_1,\phi_2 \in \cS(\rd)$; then $W(f,g)\in \cS'(\rdd,\cL^1(H))$ is such that
\[ \lc \la W(f,g),\Phi \ra u, v \rc_H = \lc \la f,\phi_1\ra,v\rc_H \overline{\lc\la g,\phi_2\ra,u\rc_H}, \quad u,v\in H. \]

Assume now $\sigma \in \cS'(\rdd,\cL(H))$. The Weyl transform $\sigma^{\w} : \cS(\rd,H)\to \cS'(\rd,H)$ is defined by duality as 
\begin{equation} \la \sigma^{\w}f,g \ra = \int_{\rdd}\tr\left[ \sigma(x,\xi) W(g,f)(x,\xi) \right]dxd\xi, \quad f,g\in \cS(\rd,H). \end{equation}
For further details see \cite[pp. 135--137]{folland} and \cite{wahlberg}. 

A classical, remarkable result in the scalar-valued case is the boundedness of Weyl transforms with symbols in the Sj\"ostrand class on any modulation and Wiener amalgam space \cite[Thm. 14.5.2]{gro book}. This property still holds in the vector valued case.  

\begin{theorem} \label{bound weyl} Let $1 \le p,q \le \infty$, $\gamma \ge 0$ and $r,s \in \bR$ be such that $|r|+|s| \le \gamma$; denote by $X$ any of the spaces $\cM^{p,q}_{r,s}(\rd,H)$ or $\cW^{p,q}_{r,s}(\rd,H)$. If $\sigma \in M^{\infty,1}_{0,2\gamma}(\rdd,\cL(H))$ then the Weyl operator $\sigma^{\w}$ is bounded on $X$.  
\end{theorem}
\begin{proof}
	The case $X=\cM^{p,q}_{r,s}(H)$ is covered by \cite[Cor. 4.8]{wahlberg}, and it is stated here with small modifications in the spirit of \cite[Thm. 14.5.6]{gro book} in order to take the weights into account. 
	For the case $X=\cW^{p,q}_{r,s}(H)$ we need an extension of the well-known \textit{symplectic covariance} property of the Weyl calculus \cite{dg book,gro book}, namely \[ \cF \sigma^{\w} = \sigma_{J^{-1}}^{\w} \cF, \qquad \sigma \in \cS'(\rdd,\cL(H)), \] where $\sigma_{J^{-1}}= \sigma \circ J^{-1}$; the proof is a straightforward application of Proposition \ref{wig prop} above. In view of this property, consider the following diagram: 
	\[
	\xymatrix{\cM_{r,s}^{p,q}(\rd,H)\ar[r]^{\sigma_J^{\w}} &  \cM_{r,s}^{p,q}(\rd,H)\ar[d]^{\mathcal{F}}\\
		\cW^{p,q}_{r,s}(\rd,H) \ar[r]^{\sigma^{\w}}\ar[u]_{\mathcal{F}^{-1}} &  \cW^{p,q}_{r,s}(\rd,H)
	}
	\] 
	It is easy to prove that if $\sigma \in M^{\infty,1}_{0,2\gamma}(\rdd,\cL(H))$ then $\sigma_J \in M^{\infty,1}_{0,2\gamma}(\rdd,\cL(H))$ too (cf. for instance the proof of \cite[Lem. 5.2]{cnt 18}), hence the preceding case implies that $\sigma_J^{\w}$ is bounded on $\cM^{p,q}_{r,s}(\rd,H)$ for any $1\le p,q \le \infty$ and $r,s\in\bR$ such that $|r|+|s|\le \gamma$. 
\end{proof}

The relevance of the Sj\"ostrand class is also enforced by the following characterization - the proof goes exactly as that of \cite[Thm. 14.5.3]{gro book} and \cite[Lem. 6.1]{gro rze} with $|\cdot|$ replaced by $|\cdot |_{E}$.
\begin{proposition} \label{S0 in sjo} The following characterization holds: 
	\[ \S0(\rd,E) = \bigcap_{s\ge0} M^{\infty}_{0,s}(\rd,E) = \bigcap_{s\ge0} M^{\infty,1}_{0,s}(\rd,E). \]
\end{proposition}	

\begin{corollary}\label{cor S0}
Let $\sigma \in \S0(\rdd,\cL(H))$. The Weyl operator $\sigma^{\w}$ is bounded on $\cM^{p,q}_{r,s}(\rd,H)$ for any $1\le p,q \le \infty$ and $r,s \in \bR$. 
\end{corollary}
\subsection{Narrow convergence} Convergence in $M^{\infty,1}$ norm is a very strong requirement. For instance it is well known that $C^{\infty}_c$ is not dense $M^{\infty,1}$ with the norm topology \cite{sjo}; this fact inhibits the standard approximation arguments and leads to restrict to subspaces such as $\cM^{\infty,1}$. Another way to cope with this problem consists in weakening the notion of convergence as follows \cite{cnr rough 6,toft cont 1}.

\begin{definition}\label{nar conv} 
Let $\Omega$ be a subset of some Euclidean space and $s\in\bR$. The map $\Omega \ni \nu \mapsto \sigma_\nu \in M^{\infty,1}_{0,s}(\rd,E)$ is said to be continuous for the narrow convergence if: 
\begin{enumerate}
\item it is a continuous map in $\cS'(\rd,E)$ (weakly), and
\item there exists a function $h \in L^1_s(\rd)$ such that for some (hence any) nonzero window $g\in \cS(\rd)$ one has $\sup_{z\in\rd}| V_g\sigma_\nu(x,\xi)|_E \le h(\xi)$ for any $\nu \in \Omega$ and a.e. $\xi \in \rd$.
\end{enumerate}
\end{definition}

The benefits of narrow continuity in the scalar-valued case carry over to the Hilbert-valued case. The following property will be used below. 

\begin{theorem} \label{nar conv weyl}
For any $1 \le p,q \le \infty$ and $\gamma \ge 0$, $r,s \in \bR$ such that $|r|+|s| \le \gamma$, let $X$ denote either $\cM^{p,q}_{r,s}(\rd,H)$ or $\cW^{p,q}_{r,s}(\rd,H)$. If $\Omega \ni \nu \mapsto \sigma_\nu \in M^{\infty,1}_{0,2\gamma}(\rd,\cL(H))$ is continuous for the narrow convergence then the corresponding map of operators $\nu \mapsto \sigma_\nu^\w$ is strongly continuous on $X$. \end{theorem}

\begin{proof}
The proof for $X=\cM^{p,q}_{r,s}(\rd,H)$ is a suitable adaption of the one given in \cite[Prop. 3]{cnr rough 6}. For the strong continuity on $X=\cW^{p,q}_{r,s}(\rd,H)$ we reduce to the latter case by the same arguments in the proof of Proposition \ref{bound weyl}, which imply that $\sigma_\nu^\w u = \cF (\sigma_\nu)_J^\w \cF^{-1}u$ for $u \in \cW^{p,q}_{r,s}(\rd,H)$. The claimed result easily follows from the continuity of the map $\nu \mapsto (\sigma_\nu)_J^\w \cF^{-1}u$ on $\cM^{p,q}_{r,s}(\rd,H)$.
\end{proof} 
\section{Estimates for the Dirac propagator}
\subsection{The free case}
Consider the Cauchy problem for the free Dirac equation, namely \eqref{dirac eq} with $V=0$: 
\begin{equation}\label{dirac free eq}
\begin{cases}
i\partial_{t}\psi\left(t,x\right)=\Dm \psi\left(t,x\right), \\
\psi\left(0,x\right)=\psi_{0}\left(x\right),
\end{cases} \qquad \left(t,x\right)\in\mathbb{R}\times\mathbb{R}^{d}.
\end{equation}
The solution can be recast in terms of the free Dirac propagator:  
\begin{equation}\label{U0} \psi(t,x)=\psi_0(x), \qquad U_0(t)=e^{-it\Dm}. \end{equation}
We can take advantage from the framework developed insofar by noticing that $U_0(t)$ is an operator-valued Fourier multiplier on the Hilbert space $H=\bC^n$, $\cL(\bC^n) \simeq \bC^{n\times n}$, with symbol 
\[ \mu_t(\xi) = \exp{\left[-2\pi i t\lc m\alpha_0 + \sum_{j=1}^d \xi_j \alpha_j \rc\right]}. \] 
An explicit expression for this matrix can be derived. After setting $C_j = - 2\pi t \xi_j$, $j=1,\ldots,d$, and $C_0 = -2 \pi tm$ we have $\mu_t(\xi) = \sum_{n\ge 0} \frac{i^n}{n!}(\sum_{j=0}^d C_j \alpha_j)^n$. The identities \eqref{dirac matrix def} satisfied by the Dirac matrices imply that 
\[ 
\begin{cases}
(\sum_{j=0}^d C_j \alpha_j)^n = (-1)^k(\sum_{j=0}^d C_j^2)^k I_n & (n=2k), \\
(\sum_{j=0}^d C_j \alpha_j)^n = i(-1)^k(\sum_{j=0}^d C_j^2)^k(\sum_{j=0}^d C_j \alpha_j) & (n=2k+1).
\end{cases}
\]
A straightforward computation finally yields

\begin{equation} \label{dirac mult} \mu_t(\xi) = \cos(2 \pi t\la \xi \ra_m) I_n -2\pi i \frac{\sin(2\pi t\la \xi \ra_m)}{2\pi \la \xi \ra_m} \lc m\alpha_0+ \sum_{j=1}^d \xi_j\alpha_j \rc, \end{equation}
from which it is clear that $\mu_t \in \S0(\rd,\mcn)$ for any fixed $t \in \bR$. 

\begin{proof}[Proof of Theorem \ref{maint free}] The proof is a direct application of Proposition \ref{fou mult mod} ($X=M^{p,q}_{r,s}(\cn)$) or Proposition \ref{fou mult was} ($X=W^{p,q}_{r,s}(\cn)$), after noticing that 
	\[ \mu_t \in \S0(\rd,\mcn) \hookrightarrow M^{\infty,1}_{0,|r|}(\rd,\mcn)\hookrightarrow W^{1,\infty}_{|r|,0}(\rd,\mcn), \quad \forall r\in \bR, \] the latter embedding being given by the Hausdorff-Young inequality \eqref{gen hy}. 
\end{proof}

\begin{proof}[Proof of estimate \eqref{C(t) estimate free}] In order to determine the time dependence of the constant $C_X(t)$, $X=M^{p,q}_{0,s}(\cn)$, we provide a different proof by making use of the discrete norm \eqref{mod disc norm} for modulation spaces. Consider the BUPU in the proof of Theorem \ref{was conv}. In view of \eqref{mod disc norm} we need to provide an estimate for $ \lV \lV \boxy_k U(t)f \rV_{L^p(\cn)} \rV_{\ell^q_s}$. We have 
	\[ \lV \boxy_k U(t)f \rV_{L^p(\cn)} = \sum_{|\ell|_{\infty} \le 1} \lV \sigma_{k+\ell} \mu_t \sigma_k \hat f \rV_{\cF L^p(\cn)} \le \sum_{|\ell|_{\infty} \le 1} \lV \sigma_{k+\ell} \mu_t \rV_{\cF L^1(\mcn)} \lV \boxy_k f \rV_{L^p(\cn)},  \]
 where we used the approximate orthogonality of the frequency-uniform decomposition operators:
\[ \boxy_k = \sum_{|\ell|_{\infty} \le 1}\boxy_k \boxy_{k+\ell}, \quad k \in \zd. \]
The multiplier estimate \eqref{bernstein est} implies
\[ \lV \sigma_{k+\ell} \mu_t \rV_{\cF L^1(\mcn)} = \lV \sigma_{0} T_{-(k+\ell)}\mu_t \rV_{\cF L^1(\mcn)} \lesssim (1+|t|)^{d/2},  \] and complex interpolation with the conservation law $ \lV \boxy_k U(t)f \rV_{L^2(\cn)} = \lV \boxy_k f \rV_{L^2(\cn)}$ yields 
\[ \lV \boxy_k U(t)f \rV_{L^p(\cn)} \lesssim (1+|t|)^{d|1/2-1/p|} \lV \boxy_k f \rV_{L^p(\cn)}. \] 
\end{proof}

This behaviour is not surprising, given that any component of a solution of the free Dirac equation is also a solution of the free Klein-Gordon equation, for which similar estimates hold \cite[Prop. 6.8]{wang book}. This connection can be exploited in many ways, as already mentioned in the Introduction; as an example one can easily prove a smoothing estimate for the free Dirac propagator. 

\begin{theorem} \label{smoothing est}
	Let $\psi(t,x)$ be the solution of \eqref{dirac free eq}. For any $t > 1$, $1\le p,q \le \infty$ and $s\in \bR$, 
\begin{equation} \label{smooth eq}
\lV \psi(t,\cdot) \rV_{M^{p,q}_{0,s}(\cn)} \lesssim \lV \psi_0 \rV_{M^{p,q}_{0,s}(\cn)} + |t|^{\gamma} \lV \psi_0 \rV_{M^{p,q}_{0,s-\gamma}(\cn)}, \qquad \gamma=d|1/2-1/p|.  
\end{equation}
\end{theorem}

\begin{proof}
Following the same strategy of \cite[Thm. 1.1]{kn},	namely projection onto the so-called positive and negative energy subspaces of the Dirac operator (cf. \cite{thaller}), it turns out that the free Dirac equation \eqref{dirac free eq} is unitarily equivalent to a pair of $(n/2)$-dimensional square-root Klein-Gordon equations, namely
\[ \begin{cases}
i\partial_{t}\psi_\pm(t,x)= \pm \la D \ra_m \psi_\pm(t,x), \\
\psi_\pm(0,x)=(\psi_{0})_\pm(x),
\end{cases} \qquad (t,x)\in\mathbb{R}\times\mathbb{R}^{d}. \]
It is then enough to replace the estimate $(3.2)$ in that paper for the Klein-Gordon semigroup $e^{it\la D \ra_m}$ with the smoothing one proved in \cite[Thm. 1.4]{deng}. The proof then proceeds in the same way.
\end{proof}
%
\subsection{The case where $V$ is a rough bounded potential} For any $1\le p,q \le \infty$, $\gamma \ge 0$ and $r,s \in \bR$ such that $|r|+|s|\le \gamma$, let $X$ denote either $\cM^{p,q}_{r,s}(\cn)$ or $\cW^{p,q}_{r,s}(\cn)$. Let $T>0$ be fixed and consider now the Cauchy problem for the Dirac equation with potential
\begin{equation}
\begin{cases}
i\partial_{t}\psi(t,x)=\left(\Dm+V(t)\right)\psi(t,x)\\
\psi(0,x)=\psi_{0}(x)
\end{cases}\qquad(t,x)\in\bR\times\rd,\label{dirac eq pot}
\end{equation}
where $V(t)=\sigma(t,\cdot)^{\w}$, $t \in [0,T]$, and the map $t\mapsto\sigma(t,\cdot)$
is continuous in $M^{\infty,1}_{0,2\gamma}(\rdd,\mcn)$ for the narrow convergence.
Standard arguments from the theory of operators semigroups (cf. \cite[Cor. 1.5]{engel}) and Theorem \ref{bound weyl} imply that for any fixed $t\in\bR$ the propagator $U(t)$ is bounded on $X$. 
\begin{proof}[Proof of Theorem \ref{maint pert}] The argument is standard, we sketch the strategy for the sake of clarity. Set $\Xi_{T}=C\left([0,T];\mathcal{L}_{s}(X)\right)$; the assumptions on $\sigma$ and Theorem \ref{nar conv weyl} imply that $V\in\Xi_{T}$. A straightforward computation shows that the propagator $U(t)$ corresponding to
	\eqref{dirac eq pot} satisfies the following Volterra integral equation: 
	\begin{equation}
	U(t)\psi_{0}=U_{0}\left(t\right)\psi_{0}-i\int_{0}^{t}U_{0}(t-s)V(s)U(s)\psi_{0}ds.\label{volterra}
	\end{equation}
A solution is given by an iterative scheme: let $\{U_n\}_{n\in\bN}$ the sequence of operators
	\[ U_0(t) \equiv e^{-it\Dm}, \qquad U_n(t)\psi_0 \coloneqq \int_0^t U_0(t-s)V(s)U_{n-1}(s)\psi_0\,ds. \]
	We have that $\{U_n\} \subset \Xi_T$, since $U_n = U_0*VU_{n-1}$ and both convolution and composition are bounded operators on $\Xi_T$; cf. \cite[Ex. 1.17.1 and Lem. B.15]{engel}. Furthermore, the following estimates hold: 
	\[ \lV U_n(t)\rV_{\cL(X)} \le K(t)^{(n+1)} \frac{t^n}{n!}, \qquad K(t)= \sup_{s \in [0,t]} \lV U_0(s)\rV \lV V(s)\rV. \]
It then follows that the Dyson-Phillips series $\sum_{n} U_n(t)$ converges with respect to the operator norm on $\cL(X)$ and also uniformly on $[0,T]$. Therefore $U(t)=\sum_n U_n(t) \in \Xi_T$ and $U(t)$ is a propagator for \eqref{dirac eq pot}. Uniqueness follows by Gronwall's lemma after noticing that a different solution $P(t)$ of \eqref{volterra} would satisfy
\[ \lV (U(t)-P(t))\psi_0\rV_{X} \le K(t) \int_0^t \lV (U(\tau)-P(\tau))\psi_0\rV_{X}d\tau. \]
\end{proof}

%

\subsection{The case where $V$ is a rough quadratic potential}

Theorem \ref{maint split mod} involves a rough potential $V$ with at most quadratic growth as in \eqref{V maint split}. A key ingredient for the proof of Theorem \ref{maint split mod} is the following lemma, which is a qualitative generalization of \cite[Lem. 3.3]{nicola trapasso}.

\begin{proposition}\label{sj decomp}
	Let $f:\rd \to E$ be such that $\partial^{\alpha}f\in M^{\infty,1}(\rd,E)$
	for any $\alpha\in\mathbb{N}^{d}$, $\left|\alpha\right|=k$ for some $k\in \bN$. Then there exist $f_{1}\in C^{\infty}_{\ge k}(\rd,E)$ and
	$f_{2}\in M^{\infty,1}(\rd,E)$ such that $f=f_{1}+f_{2}$. 
\end{proposition}
\begin{proof}
	Fix a smooth cut-off function $\chi\in C_{c}^{\infty}\left(\mathbb{R}^{d}\right)$
	supported in a neighbourhood of the origin and such that $\chi = 1$ near zero, then consider the Fourier multiplier $\chi(D)$ with symbol $\chi$. Set $f_{1}=\chi(D)f$ and $f_{2}=(I-\chi(D))f$. Clearly $f=f_{1}+f_{2}$	and we argue that $f_{1}$ and $f_{2}$ satisfy the claimed properties. 
	
	Indeed, $f_1 \in C^{\infty}(\rd,E)$ and for any $\alpha\in\mathbb{N}^{d}$, $\left|\alpha\right|=k$, we have
	\[ 	\partial^{\alpha}f_{1}=\partial^{\alpha}(\chi(D)f)=\chi(D)(\partial^{\alpha}f)\in M^{\infty,1}(\rd,E), \]
	since $\partial^{\alpha}\chi(D)$ is a Fourier multiplier with symbol $(2\pi i \xi)^{\alpha}\chi\left(\xi\right)\in C_{c}^{\infty}\left(\mathbb{R}^{d}\right)$,
	hence $\partial^{\alpha}\chi(D)=\chi(D)\partial^{\alpha}$ and $\chi(D)$ is continuous on $M^{\infty,1}(E)$ by Proposition \ref{fou mult mod}. Furthermore, similar arguments imply that for any $\alpha\in\mathbb{N}^{d}$, $\left|\alpha\right|\ge k$,
	\[\partial^{\alpha}f_{1}=\partial^{\alpha-\beta}\partial^{\beta}(\chi(D)f)=(\partial^{\alpha-\beta}\chi(D))(\partial^{\beta}f)\in M^{\infty,1}(\rd,E).	\]
	where $\beta\in\mathbb{N}^{d}$ satisfies $\left|\beta\right|=k$. 
	
	
	In order to prove the claim for $f_2$ consider the finite smooth partition
	of unity $\left\{ \varphi_{j}\right\} _{j=1}^{N}$ of the unit sphere
	$S^{d-1}\subset\mathbb{R}^{d}$ subordinated to the open cover $\{U_j\}_{j=1}^d$, where \[ U_j = \{x \in S^{d-1} \,:\, x_j \neq 0 \}. \]
Then we extend each function $\varphi_{j}$ on $\mathbb{R}^{d}\setminus\left\{ 0\right\} $ by zero-degree homogeneity, namely \[ \sum_{j=1}^{d}\varphi_{j}\left(x\right)=1,\qquad\varphi_{j}\left(\alpha x\right)=\varphi_{j}\left(x\right),\qquad\forall x\in S^{d-1},\,\alpha>0. \] This procedure gives a finite partition of unity $\left\{ \varphi_{j}\right\} _{k=1}^{d}$
	on $\mathbb{R}^{d}\backslash\left\{ 0\right\} $. Then
	\begin{flalign*}
f_{2}(x) & =\int_{\rd}e^{2\pi ix\cdot\xi}(1-\chi(\xi))\hat{f}(\xi)d\xi\\
& =\sum_{j=1}^{d}\left[\int_{\rd}e^{2\pi ix\cdot\xi}\left(\frac{1-\chi(\xi)}{(2\pi i\xi_j)^{k}}\varphi_{j}(\xi)\right)\widehat{\partial_j^{k}f}(\xi)d\xi\right]\\
& =\sum_{j=1}^{d}\widetilde{\chi_j}(D)(\partial^k_j f)(x)
	\end{flalign*}
and thus $f_2 \in M^{\infty,1}(\rd,E)$ since each $\wt{\chi_j}(D)$ is a Fourier multiplier with symbol $\left(1-\chi\left(\xi\right)\right)\varphi_{j}\left(\xi\right)/(2\pi i\xi_j)^{k}\in S_{0,0}^{0}\left(\mathbb{R}^{d}\right)$,
	hence bounded on $M^{\infty,1}(\rd,E)$. 
\end{proof}

\begin{proof}[Proof of Theorem \ref{maint split mod}] 
We apply Proposition \ref{sj decomp} twice, namely to $L$ and $Q$. We get
	\begin{itemize}
	\item $L=L_1 + L_2$, where $L_1 \in C^{\infty}_{\ge 1}(\rd,\mcn)$ and $L_2 \in M^{\infty,1}(\rd,\mcn)$, and
	\item $Q=Q_1+Q_2$, where $Q_1 \in C^{\infty}_{\ge 2}(\rd)$ and $Q_2 \in M^{\infty,1}(\rd)$. 
\end{itemize}
The RHS of \eqref{dirac eq} then becomes 
\[ \cH =(\Dm + L_1 + Q_1) + (L_2 + Q_2 + \sigma^{\w}) \eqqcolon \cH_0 + V'.  \] We see that $e^{-it\cH_0}$ is a semigroup of bounded operators on $\cM^p(\rd,\bC^n)$ as a consequence of \cite[Thm. 1.2]{kn}. It is understood that we identify the multiplication by a function $f \in M^{\infty,1}(\rd,\bC)$ on $M^{p,q}(\cn)$ with the operator $fI_n \in \mcn$, hence by Remark \ref{minfty mult} we have
\[ \lV fu \rV_{M^{p,q}(\cn)} \le \lV fI_n \rV_{M^{\infty,1}(\mcn)} \lV u \rV_{M^{p,q}(\cn)} \asymp \lV f \rV_{M^{\infty,1}} \lV u \rV_{M^{p,q}(\cn)}. \]  

 The boundedness of $e^{-it\cH}$ on $\cM^p(\rd,\cn)$ then follows from the fact that $V'$ is a bounded perturbation of $\cH_0$ \cite[Cor. 1.5]{engel} by Proposition \ref{fou mult mod} and Theorem \ref{bound weyl}. 
The case where $Q=0$ follows by the same arguments. 
\end{proof} 

\section{The nonlinear equation}
A standard tool in the study of local well-posedness is the following abstract result.
\begin{theorem}[{\cite[Prop. 1.38]{tao book}}]\label{abs nonlin}
	Let $X$ and $Y$ be two Banach spaces and $D:X\to Y$ be a bounded linear operator such that 
	\begin{equation}
	\lV Du \rV_Y \le C_0 \lV u \rV_X, \label{abs cond 1}
	\end{equation} 
	for all $u \in X$ and some $C_0 >0$. Consider then a nonlinear operator $F: Y\to X$, $F(0)=0$, such that
	\begin{equation}
	\lV F(u)-F(v) \rV_X \le \frac{1}{2C_0} \lV u-v \rV_Y, \label{abs cond 2}
	\end{equation} 
	for all $u,v$ in the ball $B_{\epsilon}(0) = \{u \in Y : \lV u \rV_Y \le \epsilon \}$ for some $\epsilon >0$. 
	Then for any $u_0 \in B_{\epsilon/2}$ there exists a unique solution $u \in B_{\epsilon}$ to the equation 
	\[ u=u_0 + DF(u), \]
	and the map $u_0 \mapsto u$ is Lipschitz with constant at most $2$, that is $\lV u \rV_Y \le 2 \lV u_0 \rV_Y$. 
\end{theorem}


With that in mind, for the sake of clarity we anticipate some estimates for the nonlinearity \eqref{gen nonlin}.

\begin{lemma}\label{nonlin gen lemma}
		Let $r,s\ge0$, $1\le p \le \infty$ and $\epsilon >0$, and consider a nonlinear function $F$ as in \eqref{gen nonlin}. Denote by $X$ any of the spaces $M^{p,1}_{0,s}(\cn)$ or $W^{1,p}_{r,s}(\cn)$. If $\psi_0 \in X$ then $F(\psi) \in X$ and, for any $\psi,\phi \in B_{\epsilon}(0)\subset X$ there exists a constant $C_{\epsilon} >0$ such that 
		\[ \lV F(\psi) - F(\phi) \rV_{X} \le C_{\epsilon} \lV \psi - \phi \rV_{X}. \] 
\end{lemma}
\begin{proof}
	In view of Proposition \eqref{mod sp prop} $(iv)$ and its counterpart for amalgam spaces the first claim is an easy consequence of the algebra property of $X$ under pointwise multiplication \cite[Lem. 2.1-2.2]{cn wave 3} and the series expansion of each component. The estimate in the second part follows from a straightforward computation (cf. the proof of \cite[Thm. 4.1]{cn wave 3}), that is 
	\begin{align*} F_j(\psi)-F_j(\phi) & = \int_0^1 \frac{d}{dt} F_j(t \psi +(1-t)\phi) dt \\
	& = \sum_{k=1}^n \left[ (\psi_k-\phi_k)\sum_{\alpha,\beta,\gamma,\delta \in \bN^n} c^{j,k}_{\alpha,\beta,\gamma,\delta} \psi^\alpha \bar{\psi}^\beta \phi^\delta \bar{\phi}^\gamma \right. \\ & \left. + (\bar{\psi}_k - \bar{\phi}_k)\sum_{\alpha,\beta,\gamma,\delta \in \bN^n} \tilde{c}^{j,k}_{\alpha,\beta,\gamma,\delta \in \bN^n} \psi^{\alpha} \bar{\psi}^{\beta} \phi^{\delta} \bar{\phi}^{\gamma} \right]. \end{align*}
	Again by Proposition \ref{mod sp prop} $(iv)$ we have 
	\[ \lV F(\psi) - F(\phi) \rV_{X}  \lesssim \lV \psi - \phi \rV_{X}  \sum_{j,k=1}^n \sum_{\alpha,\beta,\gamma,\delta \in \bN^n} C^{j,k}_{\alpha,\beta,\gamma,\delta} \lV \psi\rV_{X}^{|\alpha+\beta|}\lV \phi\rV_{X}^{|\gamma+\delta|}, \]
	with $C^{j,k}_{\alpha,\beta,\gamma,\delta} = |c^{j,k}_{\alpha,\beta,\gamma,\delta}|+|\tilde{c}^{j,k}_{\alpha,\beta,\gamma,\delta}|$, and the latter expression is $\le C_{\epsilon} \lV \psi - \phi \rV_{X}$ whenever $\psi,\phi \in B_{\epsilon}(0)$. 
	\end{proof}

\begin{proof}[Proof of Theorem \ref{maint nonlin}] The proof is an application of the iteration scheme given in Theorem \ref{abs nonlin}. In particular we choose either $X = M^{p,1}_{0,s}(\cn)$ or $X=W^{1,p}_{r,s}(\cn)$, then $Y= C^0([0,T],X)$, and convert \eqref{dirac eq nonlin} in integral form:
	\[ \psi(t) = U_0(t)\psi_0 -i\int_0^t U_0(t-s)F(\psi(s))ds, \] where $U_0=e^{-it\cD_m}$ is the free propagator. It is then enough to prove \eqref{abs cond 1} and \eqref{abs cond 2} in this setting, where $D$ is the Duhamel operator $D = \int_0^t U_0(t-s) \cdot ds$. 
	First, notice that from Theorem \ref{maint free} we have that 
	\[ \lV U_0(t)\psi_0\rV_{X} \le C_T \lV \psi_0 \rV_{X}, \quad \forall t \in [0,T]. \] Therefore, \[ \lV \int_0^t U_0(t-s)u(s)ds \rV_{X} \le \int_0^t \lV U_0(t-s)u(s) \rV_{X}ds  \le T C_T \sup_{t \in [0,T]} \lV u(t)\rV_{X}.   \]
	Lemma \ref{nonlin gen lemma} then provides \eqref{abs cond 1} with a constant $C_0 = O(T)$ and also \eqref{abs cond 2}. The claim follows after choosing $T=T(\lV \psi_0 \rV_{X})$ sufficiently small. \end{proof}
	
	

\begin{remark} 
	A more general version of Theorem \ref{maint nonlin}, namely a nonlinear variant of Theorem \ref{maint pert}, can be stated. For any $1\le p \le \infty$ and $\gamma\ge 0$ let $X$ denote either $\cM^{p,1}_{0,s}(\cn)$ with $0\le s \le \gamma$ or $\cW^{1,p}_{r,s}(\cn)$ with $r,s\ge0$ such that $r+s \le \gamma$. The differential operator $L= i\partial_t - \cD_m$ in \eqref{dirac eq nonlin}, namely $L\psi = F(\psi)$,  is now extended to $L= i\partial_t - \cD_m - \sigma_t^{\w}$, where the symbol map $[0,T]\ni t \mapsto \sigma(t,\cdot) \in M^{\infty,1}_{0,2\gamma}(\mcn)$ is continuous for the narrow convergence and the nonlinear term is \eqref{gen nonlin}. We recast the problem in integral form as 
	\[ \psi(t) = U(t,0)\psi_0 - i\int_0^t U(t,\tau)F(\psi(\tau))d\tau, \] where $U(t,\tau)$, $0 \le \tau \le t \le T$ is the linear propagator constructed in the proof of Theorem \ref{maint pert} corresponding to initial data at time $\tau$. In order for the iteration scheme in Theorem \ref{abs nonlin} to work it is enough to prove that $U(t,\tau)$ is strongly continuous on $X$ jointly in $(t,\tau)$, $0\le \tau \le t \le T$; the latter condition would imply a uniform bound for the operator norm with respect to $t,\tau$ as a consequence of the uniform boundedness principle. 
	Theorem \ref{maint pert} yields strong continuity of $U(t,\tau)$ in $t$ for fixed $\tau$. The time-reversibility enjoyed by the equation implies that the same holds after switching $\tau$ and $t$. Furthermore, for $\tau' \le \tau \le t$ we have
	\begin{align*} \lV U(t,\tau)\psi_0 - U(t,\tau')\psi_0 \rV_{X} \le & C \lV \psi_0 - U(\tau,\tau')\psi_0 \rV_{X}, 
	\end{align*}
	hence the map $\tau \mapsto U(t,\tau)\psi_0$ is continuous in $X$, uniformly with respect to $t$ and this gives the desired result. 
\end{remark}

\textbf{Acknowledgments.} The author gratefully thanks Professor Fabio Nicola for fruitful discussions and constant support.


\begin{thebibliography}{99}
	
\bibitem{amann quasilin 1} Amann, Herbert. \textit{Linear and quasilinear parabolic problems. Vol. I. Abstract linear theory.} Monographs in Mathematics, 89. Birkh\"auser Boston, 1995. 


\bibitem{amann book 3} Amann, Herbert. \textit{Linear and quasilinear parabolic problems. Vol. II. Function spaces.} Monographs in Mathematics, 106. Birkh\"auser Basel, 2019. 

\bibitem{arendt} Arendt, Wolfgang; Batty, Charles J. K.; Hieber, Matthias; Neubrander, Frank. \textit{Vector-valued Laplace transforms and Cauchy problems.} Second edition. Monographs in Mathematics, 96. Birkhäuser/Springer Basel AG, Basel, 2011. 

\bibitem{bejenaru} Bejenaru, Ioan; Herr, Sebastian. The cubic Dirac equation: small initial data in $H^{\frac 12}(\bR^2)$. \textit{Comm. Math. Phys.} \textbf{343} (2016), no. 2, 515--562. 

\bibitem{benyi 1} B\'enyi, \'Arp\'ad; Okoudjou, Kasso A. Local well-posedness of nonlinear dispersive equations on modulation spaces. \textit{Bull. Lond. Math. Soc.} \textbf{41} (2009), no. 3, 549--558. 

\bibitem{benyi unimod 2} B\'enyi, \'Arp\'ad; Gr\"ochenig, Karlheinz; Okoudjou, Kasso A.; Rogers, Luke G. Unimodular Fourier multipliers for modulation spaces. \textit{J. Funct. Anal.} \textbf{246} (2007), no. 2, 366--384. 


\bibitem{cacciafesta da 2} Cacciafesta, Federico; D'Ancona, Piero. Endpoint estimates and global existence for the nonlinear Dirac equation with potential. \textit{J. Differential Equations} \textbf{254} (2013), no. 5, 2233--2260. 

\bibitem{cacciafesta f 3} Cacciafesta, Federico; Fanelli, Luca. Dispersive estimates for the Dirac equation in an Aharonov-Bohm field. \textit{J. Differential Equations} \textbf{263} (2017), no. 7, 4382--4399. 


\bibitem{cheng} Chen, Jiecheng; Fan, Dashan. Estimates for wave and Klein-Gordon equations on modulation spaces. \textit{Sci. China Math.} \textbf{55} (2012), no. 10, 2109--2123. 



\bibitem{cn strichartz 2} Cordero, Elena; Nicola, Fabio. Some new Strichartz estimates for the Schr\"odinger equation. \textit{J. Differential Equations} \textbf{245} (2008), no. 7, 1945--1974.

\bibitem{cn wave 3} Cordero, Elena; Nicola, Fabio. Remarks on Fourier multipliers and applications to the wave equation. \textit{J. Math. Anal. Appl.} \textbf{353} (2009), no. 2, 583--591. 

\bibitem{cn pot mod 4} Cordero, Elena; Nicola, Fabio. On the Schr\"odinger equation with potential in modulation spaces. \textit{J. Pseudo-Differ. Oper. Appl.} \textbf{5} (2014), no. 3, 319--341. 



\bibitem{cnr rough 6} Cordero, Elena; Nicola, Fabio; Rodino, Luigi. Schr\"odinger equations with rough Hamiltonians. \textit{Discrete Contin. Dyn. Syst.} \textbf{35} (2015), no. 10, 4805--4821. 

\bibitem{cnt 18} Cordero, Elena; Nicola, Fabio; Trapasso, S. Ivan. Almost Diagonalization of $\tau $-Pseudodifferential Operators with Symbols in Wiener Amalgam and Modulation Spaces. \textit{J. Fourier Anal. Appl.} \textbf{25} (2019), no. 4, 1927--1957. 


\bibitem{d'ancona} D'Ancona, Piero; Fanelli, Luca. Decay estimates for the wave and Dirac equations with a magnetic potential. \textit{Comm. Pure Appl. Math.} \textbf{60} (2007), no. 3, 357--392.

\bibitem{dg book} de Gosson, Maurice A. \textit{Symplectic methods in harmonic analysis and in mathematical physics}. Pseudo-Differential Operators. Theory and Applications, 7. Birkh\"auser/Springer Basel AG, Basel, 2011. 

\bibitem{deng} Deng, Qingquan; Ding, Yong; Sun, Lijing. Estimate for generalized unimodular multipliers on modulation spaces. \textit{Nonlinear Anal.} \textbf{85} (2013), 78--92. 

\bibitem{engel} Engel, Klaus-Jochen; Nagel, Rainer. \textit{A short course on operator semigroups.} Universitext. Springer, New York, 2006. 

\bibitem{erdogan} Erdo\v{g}an, M. Burak; Goldberg, Michael; Green, William R. Limiting absorption principle and Strichartz estimates for Dirac operators in two and higher dimensions. \textit{Comm. Math. Phys.} 367 (2019), no. 1, 241--263. 

\bibitem{fei new segal} Feichtinger, Hans G. On a new Segal algebra. \textit{Monatsh. Math.} \textbf{92}(4) (1981), 269--289.

\bibitem{fei modulation 83} Feichtinger, Hans G. Modulation spaces on locally compact abelian groups. Technical report, University of Vienna, 1983.

\bibitem{fei wien 83} Feichtinger, Hans G. Banach convolution algebras of Wiener type. In \textit{Functions, series, operators, Vol. I, II (Budapest, 1980)}, 509--524, Colloq. Math. Soc. J\'anos Bolyai, 35, North-Holland, Amsterdam, 1983.

\bibitem{fei dec 2} Feichtinger, Hans G. Banach spaces of distributions defined by decomposition methods. II. \textit{Math. Nachr.} \textbf{132} (1987), 207--237. 

\bibitem{fei looking} Feichtinger, Hans G. Modulation spaces: looking back and ahead. \textit{Sampl. Theory Signal Image Process.} \textbf{5} (2006), no. 2, 109--140. 

\bibitem{fei dec 1} Feichtinger, Hans G.; Gr\"obner, Peter. Banach spaces of distributions defined by decomposition methods. I. \textit{Math. Nachr.} \textbf{123} (1985), 97--120. 

\bibitem{fei hor} Feichtinger, Hans G.; H\"ormann, Wolfgang. A distributional approach to generalized stochastic processes on locally compact Abelian groups. In \textit{New perspectives on approximation and sampling theory}, 423--446, Appl. Numer. Harmon. Anal., Birkh\"auser/Springer, Cham, 2014. 


\bibitem{folland} Folland, Gerald B. \textit{Harmonic analysis in phase space.} Annals of Mathematics Studies, 122. Princeton University Press, Princeton, NJ, 1989. 

\bibitem{girardi class 1} Girardi, Maria; Weis, Lutz. Vector-valued extensions of some classical theorems in harmonic analysis. In \textit{Analysis and applications—ISAAC 2001 (Berlin)}, 171--185, Int. Soc. Anal. Appl. Comput., 10, Kluwer Acad. Publ., Dordrecht, 2003.

%
%


\bibitem{gro sjo} Gr\"ochenig, Karlheinz. Time-frequency analysis of Sj\"ostrand's class. \textit{Rev. Mat. Iberoam.} \textbf{22} (2006), no. 2, 703--724.

\bibitem{gro book} Gr\"ochenig, Karlheinz. \textit{Foundations of time-frequency analysis.} Appl. Numer. Harmon. Anal., Birkh\"auser Boston, Boston, MA, 2001. 


\bibitem{gro rze} Gr\"ochenig, Karlheinz; Rzeszotnik, Ziemowit. Banach algebras of pseudodifferential operators and their almost diagonalization. \textit{Ann. Inst. Fourier (Grenoble)} \textbf{58} (2008), no. 7, 2279--2314. 

\bibitem{heil} Heil, Christopher. An introduction to weighted Wiener amalgams. In \textit{Wavelets and their Applications (Chennai, January 2002)}, M. Krishna, R. Radha and S. Thangavelu, eds., Allied Publishers, New Dehli (2003), pp. 183--216.

\bibitem{horm 3} H\"ormander, Lars. \textit{The analysis of linear partial differential operators. III. Pseudo-differential operators.} Reprint of the 1994 edition. Classics in Mathematics. Springer, Berlin, 2007. 

\bibitem{huh 1} Huh, Hyungjin. Global strong solution to the Thirring model in critical space. \textit{J. Math. Anal. Appl.} \textbf{381} (2011), no. 2, 513--520.

\bibitem{hytonen book 1} Hyt\"onen, Tuomas; van Neerven, Jan; Veraar, Mark; Weis, Lutz. \textit{Analysis in Banach spaces. Vol. I. Martingales and Littlewood-Paley theory.} A Series of Modern Surveys in Mathematics, 63. Springer, Cham, 2016.

\bibitem{hytonen 2} Hyt\"onen, Tuomas; Portal, Pierre. Vector-valued multiparameter singular integrals and pseudodifferential operators. \textit{Adv. Math.} \textbf{217} (2008), no. 2, 519--536.

\bibitem{kalf} Kalf, Hubert; Yamada, Osanobu. Essential self-adjointness of $n$-dimensional Dirac operators with a variable mass term. \textit{J. Math. Phys.} \textbf{42} (2001), no. 6, 2667--2676. 

\bibitem{kato ki 1} Kato, Keiichi; Kobayashi, Masaharu; Ito, Shingo. Representation of Schrödinger operator of a free particle via short-time Fourier transform and its applications. \textit{Tohoku Math. J. (2)} \textbf{64} (2012), no. 2, 223--231. 


\bibitem{kato ki 3} Kato, Keiichi; Kobayashi, Masaharu; Ito, Shingo. Estimates on modulation spaces for Schrödinger evolution operators with quadratic and sub-quadratic potentials. \textit{J. Funct. Anal.} \textbf{266} (2014), no. 2, 733--753. 

\bibitem{kn} Kato, Keiichi; Naumkin, Ivan. Estimates on the modulation spaces for the Dirac equation with potential. \textit{Rev. Mat. Complut.} \textbf{32} (2019), no. 2, 305--325.

\bibitem{kerman} Kerman, Ronald A. Convolution theorems with weights. \textit{Trans. Amer. Math. Soc.} \textbf{280} (1983), no. 1, 207--219. 

\bibitem{kwap} Kwapie\'n, Stanis\l{}aw. Isomorphic characterizations of inner product spaces by orthogonal series with vector valued coefficients. \textit{Studia Math.} \textbf{44} (1972), 583--595. 

\bibitem{machihara 1} Machihara, Shuji; Nakanishi, Kenji; Ozawa, Tohru. Small global solutions and the nonrelativistic limit for the nonlinear Dirac equation. \textit{Rev. Mat. Iberoamericana} \textbf{19} (2003), no. 1, 179--194. 



\bibitem{naumkin} Naumkin, I. P. Initial-boundary value problem for the one dimensional Thirring model. \textit{J. Differential Equations} \textbf{261} (2016), no. 8, 4486--4523. 


\bibitem{nicola trapasso} Nicola, Fabio; Trapasso, S. Ivan. On the pointwise convergence of the integral kernels in the Feynman-Trotter formula. To appear in \textit{Comm. Math. Phys.}, 2019. 

\bibitem{ozawa} Ozawa, Tohru; Yamauchi, Kazuyuki. Structure of Dirac matrices and invariants for nonlinear Dirac equations. \textit{Differential Integral Equations} \textbf{17} (2004), no. 9-10, 971--982. 

\bibitem{pelinovsky} Pelinovsky, Dmitry. Survey on global existence in the nonlinear Dirac equations in one spatial dimension. \textit{Harmonic analysis and nonlinear partial differential equations}, 37--50, RIMS K\^oky\^uroku Bessatsu, B26, Res. Inst. Math. Sci. (RIMS), Kyoto, 2011.



\bibitem{reich} Reich, Maximilian; Sickel, Winfried. Multiplication and composition in weighted modulation spaces. In \textit{Mathematical analysis, probability and applications—plenary lectures}, 103--149, Springer Proc. Math. Stat., 177, Springer, Cham,  2016. 

%
%
%

\bibitem{sjo} Sj\"ostrand, Johannes. An algebra of pseudodifferential operators. \textit{Math. Res. Lett.} \textbf{1} (1994), no. 2, 185--192.

\bibitem{soler} Soler, Mario. Classical, stable, nonlinear spinor field with positive rest energy. \textit{Phys. Rev. D} \textbf{1} (1970), 2766-–2769.


\bibitem{sugi nonlin} Sugimoto, Mitsuru; Tomita, Naohito; Wang, Baoxiang. Remarks on nonlinear operations on modulation spaces. \textit{Integral Transforms Spec. Funct.} \textbf{22} (2011), no. 4-5, 351--358. 


\bibitem{tao book} Tao, Terence. \textit{Nonlinear dispersive equations. Local and global analysis}. CBMS Reg. Conf. Ser. Math., Amer. Math. Soc., 2006


\bibitem{thaller} Thaller, Bernd. \textit{The Dirac equation.} Texts and Monographs in Physics. Springer-Verlag, Berlin, 1992. 

\bibitem{thirring} Thirring, Walter E. A soluble relativistic field theory. \textit{Ann. Physics} \textbf{3} (1958), 91--112.

\bibitem{toft cont 1} Toft, Joachim. Continuity properties for modulation spaces, with applications to pseudo-differential calculus. I. \textit{J. Funct. Anal.} \textbf{207} (2004), no. 2, 399--429. 

\bibitem{toft cont 2} Toft, Joachim. Continuity properties for modulation spaces, with applications to pseudo-differential calculus. II. \textit{Ann. Global Anal. Geom.} \textbf{26} (2004), no. 1, 73--106.

\bibitem{wahlberg} Wahlberg, Patrik. Vector-valued modulation spaces and localization operators with operator-valued symbols. \textit{Integral Equations Operator Theory} \textbf{59} (2007), no. 1, 99--128. 

\bibitem{wang paper} Wang, Baoxiang; Hudzik, Henryk. The global Cauchy problem for the NLS and NLKG with small rough data. \textit{J. Differential Equations} \textbf{232} (2007), no. 1, 36--73. 

\bibitem{wang book} Wang, Baoxiang; Huo, Zhaohui; Hao, Chengchun; Guo, Zihua. \textit{Harmonic analysis method for nonlinear evolution equations. I.} World Scientific Publishing Co. Pte. Ltd., Hackensack, NJ, 2011. 

\bibitem{weis} Weis, Lutz. Operator-valued Fourier multiplier theorems and maximal $L_p$-regularity. \textit{Math. Ann.} \textbf{319} (2001), no. 4, 735--758. 

\bibitem{zhang} Zhao, Guoping; Chen, Jiecheng; Guo, Weichao. Klein-Gordon equations on modulation spaces. \textit{Abstr. Appl. Anal.} 2014, Art. ID 947642, 15 pp. 
\end{thebibliography}
\end{document}